\documentclass{amsart}
\usepackage{amsthm}
\usepackage{amsmath}
\usepackage{amstext}
\usepackage{amsfonts}
\usepackage{latexsym}
\usepackage{amscd}
\usepackage{xypic}
\newcommand{\la}{\ensuremath{\rightarrow}}
\theoremstyle{plain}
\newtheorem{theorem}{Theorem}[section]
\newtheorem{lemma}[theorem]{Lemma}
\newtheorem{proposition}[theorem]{Proposition}
\newtheorem{corollary}[theorem]{Corollary}

\theoremstyle{definition}
\newtheorem{definition}[theorem]{Definition}

\newtheorem{problem}[theorem]{Problem}

\newtheorem{remark}[theorem]{Remark}

\numberwithin{equation}{theorem}

\begin{document}

\title[First order deformations of n.c. varieties.]{First order deformations of schemes with normal crossing singularities.}
\author{Nikolaos Tziolas}
\address{Department of Mathematics, University of Cyprus, P.O. Box 20537, Nicosia, 1678, Cyprus}
\email{tziolas@ucy.ac.cy}




\begin{abstract}
We study the sheaf $T^1(X)$ of first order deformations of a reduced scheme with normal crossing singularities. In particular, we obtain a formula for $T^1(X)$ in a suitable log resolution of $X$.
\end{abstract}

\maketitle

\section{Introduction}
The purpose of this paper is to describe the sheaf $T^1(X)$ of first order deformations of a reduced scheme $X$ with normal crossing singularities defined over an algebraically closed field $k$. 

The motivation to study deformations of schemes with normal crossing singularities comes from many different problems. 
 
Normal crossing singularities appear quite naturally in any degeneration problem. Let $f \colon \mathcal{X} \la C$ be a flat projective morphism from a variety $\mathcal{X}$ to a curve $C$. Then, according to Mumford's semistable reduction theorem~\cite{KKMS73},  after a finite base change and a birational modification the family can be brought to standard form $f^{\prime} \colon \mathcal{X}^{\prime} \la C^{\prime}$, where $\mathcal{X}^{\prime}$ is smooth and the special fibers are simple normal crossing varieties.

Schemes with normal crossing singularities appear also in the compactification of the moduli space of varieties of general type. The boundary of the compactified moduli space of surfaces of general type consists of the stable varieties. A stable variety is a proper reduced scheme $X$ such that $X$ has only semi-log-canonical singularities and $\omega_X^{[k]}$ is locally free and ample for some $k>0$~\cite{KSB88}. The simplest class of non-normal s.l.c. singularities are the normal crossing singularities. Therefore it is interesting to know which stable varieties are smoothable and which are not. In order to do this it is essential to understand their first order deformations.

Another area where normal crossing singularities play an important role is the study of Mori fiber spaces in higher dimensional birational geometry. It is well known that the outcome of the minimal model program starting with a smooth $n$-dimensional projective variety $X$, is a $\mathbb{Q}$-factorial terminal projective variety $Y$ such that either $K_Y$ is nef, or $Y$ has a Mori fiber space structure. This means that there is a projective morphism $f \colon Y \la Z$ such that $-K_Y$ $f$-ample, $Z$ is normal and $\dim Z \leq \dim X -1$. Suppose that the second case happens and $\dim Z=1$. Let $z \in Z$ and $Y_z= f^{-1}(z)$. Then $Y_z$ is a Fano variety of dimension $n-1$ and $Y$ is a smoothing of $Y_z$. The singularities of the special fibers are difficult to describe but normal crossing singularities naturally occur and are the simplest possible non-normal singularities.

For the reasons explained previously, it is of interest to study the deformation spaces of varieties with normal crossing singularities (and in particular Fano and stable varieties). For this it is essential to study their first order deformations. 

Let $X$ be a reduced scheme with normal crossing singularities defined over a field $k$. Then $T^1(X)$ is an invertible sheaf on the singular locus $D$ of $X$. If $X$ has points of multiplicity at most 2 then $D$ is smooth. However, in general $D$ may not even be Cohen-Macauley. For this reason it is preferable to work in a smooth model of $D$ and in particular in a log resolution of $(X,D)$.

In section 3 we describe $T^1(X)$ in the case when $X$ has only double point singularities and in section 4 we treat the general case. The main results are the following.

\begin{theorem}
Let $X$ be a quasi-projective scheme with normal crossing singularities defined over an algebraically closed field $k$. Let $D$ be the singular locus of $X$. Then,
\begin{enumerate}
\item Suppose that $X$ has only singularities of multiplicity 2. Let $\pi \colon \bar{X} \rightarrow X$ be the normalization of $X$ and $\bar{D}=\pi^{-1}(D)$. Then
\begin{enumerate}
\item 
\[
T^1(X)= \omega_X \otimes \omega_D^{-1} \otimes \bigwedge^2\pi_{\ast}\mathcal{O}_{\bar{D}}(\bar{D})
\]
\item
\[
\pi^{\ast}T^1(X)=\mathcal{O}_{\bar{D}}(\bar{D})\otimes \varepsilon^{\ast}\mathcal{O}_{\bar{D}}(\bar{D})
\]
\end{enumerate}
where $\varepsilon$ is the unique involution of $\bar{D}$ over $D$ interchanging the fibers of $\pi$.
\item Suppose that $\dim X \leq 3$. Let $\phi \colon \tilde{X} \la X$ be a birational morphism such that
\begin{enumerate}
\item $(\tilde{X},1/2\tilde{D})$, $\tilde{D}$ are terminal;
\item $K_{\tilde{X}}+1/2\tilde{D}$ and $K_{\tilde{D}}$ are $\phi$-nef.
\end{enumerate}
where $\tilde{D} \subset \tilde{X}$ be the reduced divisorial part of $\phi^{-1}(D)$ that dominates $D$ (such spaces do exist). Then
\[
\phi^{\ast}T^1(X)=\mathcal{O}_{\tilde{D}}(\tilde{D}) \otimes \varepsilon^{\ast}\mathcal{O}_{\tilde{D}}(\tilde{D}) \otimes \mathcal{O}_{\tilde{X}}(3E)\otimes \mathcal{O}_{\tilde{D}}
\]
where $E\subset \tilde{X}$ is the reduced $\phi$-exceptional divisor that dominates the set of singular points of $X$ of multiplicity at least three and $\varepsilon$ the unique nontrivial involution of $\tilde{D}$ over the normalization $\bar{D}$ of $D$.
\end{enumerate}

\end{theorem}

The crucial part of the proof of the first part of the theorem is to show that if $X$ has only double points, then outside a closed set of codimension at least 2, there is a formal embeding $\hat{X} \subset \mathcal{Y}$ of the completion $\hat{X}$ of $X$ along $D$ into a smooth formal scheme $\mathcal{Y}$ of algebraic dimension equal to $\dim X +1$. This is done in Theorem~\ref{embeding}. This provides a length 2 resolution of $\hat{\Omega}_X$ which simplifies significantly the calculation of $T^1(X)$.  

In order to treat the general case, Theorem~\ref{2} gives a formula for $T^1(X)$ in the log resolution $(D^{\prime},X^{\prime})$ of $(D,X)$ that is obtained by successively blowing up the singular locus of highest multiplicity of $X$. This works in all dimensions but has the disadvantage that the log resolution used is not characterized by any numerical property making it unique. In dimension at most three we obtain a unique log resolution, by running an explicit  minimal model program on $(D^{\prime},X^{\prime})$ (and hence the difficulty to extend it in higher dimensions) and in this way we prove the second part of the previous theorem.  In dimension 2, $\tilde{D} \subset \tilde{X}$ is just the minimal log resolution of $D \subset X$ and the previous result is a special case of~\cite[Theorem 3.1]{Tzi09}.

\section{Preliminaries.}
This section contains the basic definitions and some technical results that will be needed for the proof of the main theorems.

All schemes in this paper are defined over an algebraically closed field $k$.

A reduced scheme $X$ of finite type over $k$ is called a normal crossing variety of dimension $n$ if for any $P \in X$, $\hat{\mathcal{O}}_{X,P} \cong k[[x_0,\dots,x_n]]/(x_0\cdots x_r)$, for some $r=r(P)$, where $\hat{\mathcal{O}}_{X,P}$ is the completion of the local ring of $X$ at $P$ at its maximal ideal. In addition, if $X$ has smooth irreducible components then it is called a simple normal crossing variety. 

The algebraically closedness assumption is not essential and is only imposed in order to have a simple definition of normal crossing singularities. I believe that in the case of of a non-algebraically closed field $k$, a point $P \in X$ should be called a normal crossing singularity if and only if $\hat{\mathcal{O}}_{X,P} \otimes_L \bar{L} \cong L[[x_0,\dots,x_n]]/(x_0\cdots x_r)$, where $L$ is a coefficient field of $\mathcal{O}_{X,P}$ and $\bar{L}$ its algebraic closure. This way singularities like $x_0^1+x_1^2=0$ over $\mathbb{R}$ will be called normal crossing.

For any scheme $X$ we denote by $T^1(X)$ the sheaf of first order deformations of $X$~\cite{Sch68}. If $X$ is reduced then $T^1(X)=\mathcal{E}xt^1_X(\Omega_X,\mathcal{O}_X)$.

A formal scheme $\mathcal{X}=(X,\mathcal{O}_{\mathcal{X}})$ essentially of finite type over $k$ is called smooth, if and only if, for any $P\in X$, $\mathcal{O}_{\mathcal{X},P}$ is a regular local ring.

Let $\mathcal{X}$ be a smooth formal scheme of algebraic dimension $n$. Then we define the formal dualizing sheaf $\omega_{\mathcal{X}}$ of $\mathcal{X}$ to be the invertible sheaf $\wedge^n\hat{\Omega}_{\mathcal{X}}$.

Let $X$ be a scheme having a dualizing sheaf $\omega_X$ and $Y \subset X$ a closed subscheme of $X$. Let $\hat{X}$ be the formal completion of $X$ along $Y$. Then we define $\omega_{\hat{X}}=\hat{\omega}_X$. If $X$ is smooth then the two previous definitions of formal dualizing sheaf agree. 

It is easy to show, either by direct calculation or by refering to~\cite{LRT07} and~\cite{LNS05}, that the formal dualizing sheaves satisfy similar adjunction properties for embedings of formal schemes as in the case of usual schemes. For details and basic properties of formal schemes and differentials, we refer the reader to~\cite{LRT07} and~\cite{LNS05}. At this point I would like to mention that one could define a theory of dualizing sheaves for general formal schemes and prove results similar to the scheme case, but for the purpose of this paper it is not necessary to do so and we avoid the complications of the general case by treating only  the case of either a smooth formal scheme or an algebraizable one that will be used in the next section.  

The following result of Artin provides a bridge between the simple normal crossing and the normal crossing case.

\begin{theorem}[~\cite{Art69}]\label{artin}
Let $X_1$, $X_2$ be $S$-schemes of finite type, and let $x_i \in X_i$, $i=1,2$, be points. If $\hat{\mathcal{O}}_{X_1,x_1} \cong \hat{\mathcal{O}}_{X_2,x_2}$, then
$X_1$ and $X_2$ are locally isomorphic for the \'etale topology, i.e., there is a common \'etale neighborhood $(X^{\prime},x^{\prime})$ of $(X_i,x_i)$, $i=1,2$. This means there is a diagram of \'etale maps
\[
\xymatrix{
  &  X^{\prime}\ar[dl]_{f_1} \ar[dr]^{f_2} & \\
X_1  &   &  X_2
}
\]
such that $f_1(x^{\prime})=x_1$, $f_2(x^{\prime})=x_2$ and inducing an isomorphism of residue fields $k(x_1)\cong k(x^{\prime}) \cong k(x_2)$
\end{theorem}

The sheaf $T^1(X)$ is supported on the singular locus of $X$ and is therefore determined by the completion $\hat{X}$ of $X$ along any subscheme $Y$ of $X$ containing $D$. The next proposition shows how to calculate $T^1(X)$ from a suitable embedding of $\hat{X}$ in a smooth formal scheme $\mathcal{Y}$.

\begin{proposition}\label{calc-of-t1}
Let  $X$ be a scheme with normal crossing singularities and let $Z \subset X$ be a closed subscheme containing the singular locus $D$ of $X$. Let $\hat{X} \subset \mathcal{Y}$ be an embedding of the completion $\hat{X}$ of $X$ along $Z$ into a smooth formal scheme $\mathcal{Y}$ of algebraic dimension equal to $\dim X +1$. Then
\[
T^1(X)=\mathcal{H}om_{\hat{X}}(\mathcal{I}_{\hat{X}}/\mathcal{I}_{\hat{X}}^2,\mathcal{O}_{\hat{X}})\otimes \mathcal{O}_D
\]
where $\mathcal{I}_{\hat{X}}$ is the ideal sheaf of $\hat{X}$ in $\mathcal{Y}$.
\end{proposition}

\begin{proof}
By~\cite{LRT07}, there is an exact sequence
\begin{equation}\label{calc-of-t1-1}
0 \rightarrow \mathcal{I}_{\hat{X}}/\mathcal{I}_{\hat{X}}^2 \rightarrow \hat{\Omega}_{\mathcal{Y}}\otimes \mathcal{O}_{\hat{X}} \rightarrow \hat{\Omega}_{\hat{X}} \rightarrow 0 
\end{equation}
where $\mathcal{I}_{\hat{X}}$ is the ideal sheaf of $\hat{X}$ in $\mathcal{Y}$. $T^1(X)$ is a line bundle on the singular locus $D$ of $X$. Since $D \subset Y$, it follows that
\begin{equation}\label{calc-of-t1-2}
T^1(X)=\mathcal{E}xt^1_X(\Omega_X,\mathcal{O}_X)=\mathcal{E}xt^1_X(\Omega_X,\mathcal{O}_X)^{\wedge}=\mathcal{E}xt^1_{\hat{X}}(\Omega_{\hat{X}},\mathcal{O}_{\hat{X}})
\end{equation}
Since $\mathcal{Y}$ is a smooth formal scheme, it follows that $\hat{\Omega}_{\mathcal{Y}}$ is locally free of rank $\dim X+1$. Moreover, $ \mathcal{I}_{\hat{X}}/\mathcal{I}_{\hat{X}}^2 $ is invertible on $\hat{X}$. Dualizing~\ref{calc-of-t1-1} we get the exact sequence
\[
\mathcal{H}om_{\hat{X}}(\mathcal{I}_{\hat{X}}/\mathcal{I}_{\hat{X}}^2 , \mathcal{O}_{\hat{X}}) \la \mathcal{E}xt^1_{\hat{X}}(\hat{\Omega}_X,\mathcal{O}_{\hat{X}}) \la \mathcal{E}xt^1_{\hat{X}}(\Omega_{\mathcal{Y}}\otimes \mathcal{O}_{\hat{X}},\mathcal{O}_{\hat{X}})=0
\]
Taking into consideration~\ref{calc-of-t1-2} and the fact that $T^1(X)$ is invertible on $D$, it follows after restricting to $D$ that
\[
T^1(X)=\mathcal{H}om_{\hat{X}}(\mathcal{I}_{\hat{X}}/\mathcal{I}_{\hat{X}}^2,\mathcal{O}_{\hat{X}})\otimes \mathcal{O}_D
\]
as claimed.
\end{proof}

\begin{corollary}\label{0}
Let $X$ be a scheme with normal crossing singularities and let $X \subset Y$ be an embedding such that $Y$ is smooth and $\dim Y = \dim X+1$. Then
\[
T^1(X)=\mathcal{O}_Y(X)\otimes \mathcal{O}_D
\]
where $D$ is the singular locus of $X$.
\end{corollary}

The previous discussion motivates the following problem.

\begin{problem}\label{problem}
Let $X$ be a scheme with nonisolated hypersurface singularities and let $D$ be its singular locus. Let $\hat{X}$ be the completion of $X$ along $D$.
\begin{enumerate}
\item Is there a smooth formal scheme $\mathcal{Y}$ of algebraic dimension $\dim X +1$ such that $\hat{X}$ is a closed formal subscheme of $\mathcal{Y}$?
\item More strongly, is there an embeding $X \subset Y$ such that $Y$ is either a smooth scheme, or algebraic space, and $\dim Y =\dim X+1$? 
\end{enumerate}
\end{problem}

The answer to~\ref{problem}.1 is yes if $X$ has normal crossing singularities of multiplicity 2. This is shown in the next section. However, I do not know if the formal embedding is induced by an actual embeding of $X$ as a divisor into a smooth scheme, or even algebraic space. In the category of complex analytic spaces, for any surface $X$ with with normal crossing singularities, there is an analytic neighborhood $U$ of its singular locus $D$ and an embeding $U \subset V$ of $U$ as a divisor in a smooth 3-fold $V$~\cite{Tzi09}. This suggests that an algebraic surface with normal crossing singularities may be embedable as a divisor into a smooth algebraic space, or that its Henselization along $D$ is embedable into a smooth Hensel scheme $Y$.

\section{The double points case.}

In this section we describe $T^1(X)$ in the case when $X$ has only normal crossing singularities of multiplicity 2. The next result together with Proposition~\ref{calc-of-t1} are essential in doing so.
\begin{theorem}\label{embeding}
Let $X$ be a quasi-projective scheme over a field $k$. Suppose that the singular points of $X$ are normal crossing singularities of multiplicity two and let $D$ be its singular locus. Then 
there is an open set $U\subset X$ such that $\mathrm{codim} (D-D\cap U,D)\geq 2$ and an embedding $\hat{U} \subset \mathcal{Y}$ of the completion of $U$ along $D\cap U$ into a smooth formal scheme $\mathcal{Y}$ of algebraic dimension equal to $\dim X +1$. Moreover, 
\[
\omega_{\hat{U}}=\omega_{\mathcal{Y}}\otimes \mathcal{N}_{\hat{U}/\mathcal{Y}}
\]
where $\mathcal{N}_{\hat{U}/\mathcal{Y}}=\mathcal{H}om_{\hat{U}}(\mathcal{I}_{\hat{U}}/\mathcal{I}_{\hat{U}}^2,\mathcal{O}_{\hat{U}})$, and $\mathcal{I}_{\hat{U}}$ is the ideal sheaf of $\hat{U}$ in $\mathcal{Y}$.
\end{theorem}

\begin{proof}
Since $X$ is quasi-projective, there is an open immersion $i \colon X \rightarrow Y$, where $Y$ is a projective variety. Let $A$ be a very ample divisor on $Y$ and $H_1$, $H_2$ two general members of it. 
Then $Y-H_1$ and $Y-H_2$ are affine open subets of $Y$. Hence $U_1=i^{-1}(Y-H_1)$ and $U_2=i^{-1}(Y-H_2)$ are affine open subsets of $X$. Moreover, since $H_1$ and $H_2$ are general, 
\[
\begin{array}{lcr}
\mathrm{codim}(X-U_1\cup U_2,X) \geq 2 & \text{and}  & \mathrm{codim}(D-(U_1\cup U_2)\cap D,D) \geq 2
\end{array}
\]
For the remainder of the proof we may assume that $X=U_1\cup U_2$, where $U_1$ and $U_2$ are affine with hypersurface singularities. Let $U_i \subset V_i$ be an embeding of $U_i$ into a smooth variety $V_i$ such that $\dim V_i =\dim U_i +1$, $i=1,2$. Let $\hat{U}_i$, $\hat{V}_i$ be the completions of $U_i$ and $V_i$ along $D\cap U_i$, respectively. Then $\hat{U}_i$ is a closed formal subscheme of $\hat{V}_i$, $i=1,2$. We want to glue $\hat{V}_i$ into a formal scheme $\mathcal{Y}$. This will follow from the next claim.

\textbf{Claim:} Let $X$ be an affine scheme with normal crossing singularities of multiplicity two. Let $X \subset Y_i$, $i=1,2$, two embedings of $X$ in two smooth affine schemes $Y_1$ and $Y_2$ such that $\dim Y_i=\dim X+1$, $i=1,2$. Then $\hat{Y}_1 \cong \hat{Y}_2$, where $\hat{Y}_i$ are the completions of $Y_i$ along the singular locus $D$ of $X$.

We proceed to show the claim. Suppose $X = \mathrm{Spec} A$ and $Y_i = \mathrm{Spec} B_i$. Let $I=I_{D,X}$ and $Q_i=I_{D,Y_i}$ be the ideals of $D$ in $A$ and $B_i$, $i=1,2$, respectively. A local calculation shows that $D$ is smooth and that $I/I^2$ is free of rank 2 on $\mathcal{O}_D=A/I$. Since $D$ is smooth, $Q_i/Q_i^2$ are also free of rank 2 on $\mathcal{O}_D$. Hence from the natural surjections $Q_i/Q_i^2 \rightarrow I/I^2$, it follows that $Q_i/Q_i^2=I/I^2$ and therefore  $B_i/Q_i^2 \cong A/I^2$, $i=1,2$. Next we show that there are compatible isomorphisms $\phi_n \colon B_1/Q_1^{n+1} \rightarrow B_2/Q_2^{n+1}$, for all $n \geq  0$. We do this by induction on $n$. For $n=0$, this is trivial. Assume that for all $k \leq n-1$, there are compatible isomorphisms $\theta_k \colon B_1/Q_1^{k+1} \rightarrow B_2/Q_2^{k+1}$. Consider the square zero extension
\[
0 \rightarrow Q_2/Q_2^{n+1} \rightarrow B_2/Q_2^{n+1} \rightarrow B_2/Q_2^n \rightarrow 0
\]
Let $\theta^{\prime}_{n-1}\colon B_1 \rightarrow B_2/Q_2^{n}$ be the composition of $\theta_{n-1}$ with the projection $B_1 \rightarrow B_1/Q_1^n$. Then since $B_1$ is smooth, the infinitesimal lifting property for $B_1$ gives a lifting $\theta_n^{\prime} \colon B_1 \rightarrow B_2/Q_2^{n+1}$ of $\theta^{\prime}_n$. Since the maps are compatible, it follows that $\theta^{\prime}_n(Q_1) \subset Q_2/Q_2^{n+1}$. Hence $\theta^{\prime}_n $ factors to a homomorphism $\theta_n \colon B_1/Q_1^{n+1} \rightarrow B_2/Q_2^{n+1}$, giving a commutative diagram of square zero extensions 
\[
\xymatrix{
0 \ar[r]  &  Q_1^{n}/Q_1^{n+1} \ar[r]\ar[d]^{\psi_n} & B_1/Q_1^{n+1} \ar[r]\ar[d]^{\theta_n} & B_1/Q_1^{n+1} \ar[d]^{\theta_{n-1}}\ar[r] & 0 \\
0 \ar[r] & Q_2^{n}/Q_2^{n+1} \ar[r]  & B_2/Q_2^{n+1} \ar[r] &  B_2/Q_2^{n+1} \ar[r] & 0
}
\]
Next we claim that $\theta_n$ is surjective. Then from the above commutative diagram, and since $\theta_{n-1}$ is an isomorphism, it follows that $\psi_n$ is surjective too. But $Q_1^{n}/Q_1^{n+1}$ and $Q_1^{n}/Q_1^{n+1}$ are both free of the same rank (both isomorphic to $S^n(I/I^2)$) and hence $\psi_n$ is in fact an isomorphism. Therefore $\theta_n$ is an isomorphism as well, which concludes the proof of the theorem. The fact that $\theta_n$ is surjective follows from the next lemma. Finally, the adjunction formula stated, follows by completing the usual adjunctions in the embeddings $U_i \subset V_i$ and glueing them. 
\end{proof}

\begin{lemma}
Let $f \colon A \rightarrow B$ be a ring homomorphism of Noetherian rings. Let $I \subset A$ and $J \subset B$ be ideals in $A$ and $B$ respectively, such that $f(I)\subset J$ and both $I$ and $J$ are contained in the Jacobson radicals of $A$ and $B$, respectively. Moreover, suppose that
\begin{enumerate}
\item $f$ induces isomorphisms 
\begin{gather*}
A/I \rightarrow B/J \\
I/I^2 \rightarrow J/J^2
\end{gather*}
\item $B$ is finitely generated as an $A$-module.
\end{enumerate}
Then $f$ is surjective.  
\end{lemma}
\begin{proof}
From the hypothesis of the lemma it follows that $J=f(I)B+J^2$. Since $J$ is contained in the Jacobson radical of $B$, it follows from Nakayama's lemma that $f(I)B=J$. Let $N$ be the cokernel of $A \rightarrow B$, as $A$-modules. Then we get an exact sequence
\[
A/I \rightarrow B/f(I)B \rightarrow N/IN \rightarrow 0
\]
But $A/I \rightarrow B/f(I)B=B/J$  is an isomorphism by assumption. Hence $N=IN$ and since $B$ is finitely generated as an $A$-module, it follows from Nakayama's lemma that $N=0$, and hence $f$ is surjective.

\end{proof}

We will also need the following.

\begin{proposition}\label{involution}
Let $f \colon X \rightarrow Y$ be a finite \'etale morphism of degree 2 between normal varieties. Then there is a unique involution $\varepsilon \colon X \rightarrow X$ over $Y$ interchanging the fibers of $f$, i.e., for any $x\in X$, $\varepsilon(x)=x^{\prime}$, where $f^{-1}(f(x))=\{x,x^{\prime}\}$.
\end{proposition}

\begin{proof}

If $X$ is reducible then the claim is trivial. Hence we may assume that $X$, and hence $Y$, are irreducible. Let $\sigma \in \mathrm{Aut}_{K(Y)}(K(X))$ be the unique automorphism of order 2 of $K(X)$ over $K(Y)$. Let $U=\mathrm{Spec} A$ be an affine open subset of $Y$, $V=f^{-1}(U)=\mathrm{Spec} B$ and $f \colon A \rightarrow B$ the induced ring homomorphism. We will show that $\sigma \colon K(B) \rightarrow K(B)$ is induced by an $A$-homomorphism $\sigma \colon B \rightarrow B$. For this it suffices to show that $\sigma(B) \subset B$. Let $b \in B$. Since $B$ is integral over $A$, there is a relation
\[
b^n+a_{n-1}b^{n-1} + \cdots +a_1b+a_0 =0
\]
where $a_i \in A$, $0 \leq i \leq n-1$. Hence 
\[
\sigma(b)^n+a_{n-1}\sigma(b)^{n-1} + \cdots +a_1\sigma(b)+a_0 =0
\]
Therefore $\sigma(b)$ is integral over $B$. But since $B$ is normal, it follows that $b \in B$. Hence $\sigma(B) \subset B$, as claimed. Therefore $\sigma$ induces an involution $\varepsilon_V \colon V \rightarrow V$, over $U$. By the uniqueness of $\sigma$, these glue to an involution $\varepsilon \colon X \rightarrow X$ over $Y$. It remains to show that $\varepsilon$ interchanges the fibers. This follows from the following.
\begin{lemma}[~\cite{Mi80} Lemma 3.13]
Let $f, g \colon Y^{\prime} \rightarrow Y$ be $X$-morphisms, where $Y^{\prime}$ is a connected $X$-scheme and $Y$ is \'etale and separated over $X$. If there exists a point $y^{\prime} \in Y^{\prime}$ such that $f(y^{\prime})=g(y^{\prime})=y$ and the maps $k(y) \rightarrow k(y^{\prime})$ induced by $f$ and $g$ coincide, then $f=g$.
\end{lemma}
\end{proof}

\begin{theorem}\label{1}
Let $X$ be a quasi-projective scheme with only double point normal crossing singularities. Let $D\subset X$ be its singular part, $\pi \colon \tilde{X} \la X$ its normalization and $\tilde{D}=\pi^{-1}(D)$. Then, $\tilde{X}$, $\tilde{D}$ and $D$ are smooth, $\pi \colon \tilde{D} \la D$ is \'etale of degree 2, $T^1(X)$ is a line bundle on $D$ and
\begin{enumerate}
\item 
\[
T^1(X)= \omega_X \otimes \omega_D^{-1} \otimes \bigwedge^2\pi_{\ast}\mathcal{O}_{\tilde{D}}(\tilde{D})
\]
\item
\[
\pi^{\ast}T^1(X)=\mathcal{O}_{\tilde{D}}(\tilde{D})\otimes \varepsilon^{\ast}\mathcal{O}_{\tilde{D}}(\tilde{D})
\]
\end{enumerate}
where $\varepsilon$ is the unique involution of $\tilde{D}$ over $D$ interchanging the fibers of $\pi$. Moreover, $L=\omega_X \otimes \omega_D^{-1}$ is an invertible 2-torsion sheaf on $D$, i.e., $L^{\otimes 2}\cong \mathcal{O}_D$.
\end{theorem}

\begin{proof}
By Theorem~\ref{embeding} there an open subset $U \subset X$ such that $\mathrm{codim}(D-D\cap U,D) \geq 2$ and a formal embeding $\hat{U} \subset \mathcal{Y}$, where $\mathcal{Y}$ is a smooth formal scheme of algebraic dimension equal to $\dim X +1$. Since $X$ has only double point normal crossing singularities, $D$ is smooth. Moreover, $T^1(X)$ is a line bundle on $D$ and hence it is determined by its restriction to any open set of codimension $\geq 2$. Therefore by Theorem~\ref{embeding}  we may assume that there is an embeding $\hat{X} \subset \mathcal{Y}$. Then by Proposition~\ref{calc-of-t1} and standard adjunctions, it follows that
\begin{gather}\label{1.1}
T^1(X)=\mathcal{H}om_{\hat{X}}(\mathcal{I}_{\hat{X}}/\mathcal{I}_{\hat{X}}^2,\mathcal{O}_{\hat{X}})\otimes \mathcal{O}_D=\omega_{\mathcal{Y}}^{-1}\otimes \omega_{\hat{X}}=
\omega^{-1}_D \otimes \omega_{\hat{X}} \otimes \bigwedge^2 (I_{D, \mathcal{Y}}/I_{D,\mathcal{Y}}^2)
\end{gather} 
But since the completion of $X$ is along $D$, it follows that $\omega_{\hat{X}} \otimes \mathcal{O}_D=\omega_X \otimes \mathcal{O}_D$. Moreover, a straightforward local calculation shows that $I_{D,X}/I_{D,X}^2$ is locally free of rank 2 on $D$. Hence from the natural surjection $I_{D, \mathcal{Y}}/I_{D,\mathcal{Y}}^2 \rightarrow I_{D,X}/I_{D,X}^2$, it follows that 
\[
I_{D, \mathcal{Y}}/I_{D,\mathcal{Y}}^2 \cong I_{D,X}/I_{D,X}^2
\]
But $I_{D,X}/I_{D,X}^2 = \pi_{\ast} \mathcal{O}_{\tilde{D}}(\tilde{D})$~\cite{Re94}. Hence from~(\ref{1.1}) we get that
\[
T^1(X) \cong \omega_X \otimes \omega_D^{-1} \otimes \bigwedge^2\pi_{\ast}\mathcal{O}_{\tilde{D}}(\tilde{D})
\]
as claimed. Let $L=\omega_X \otimes \omega_D^{-1}$. Then by subadjunction we get that $\pi^{\ast}\omega_X =\omega_{\tilde{X}} \otimes \mathcal{O}_{\tilde{X}} (\tilde{D})$. Then, since $\tilde{D} \rightarrow D$ is \'etale, 
\[
\pi^{\ast}L=\pi^{\ast}(\omega_X \otimes \omega_D^{-1})=\omega_{\tilde{X}}\otimes \mathcal{O}_{\tilde{X}}(\tilde{D})\otimes \omega_{\tilde{D}}^{-1}=\omega_{\tilde{D}}\otimes \omega_{\tilde{D}}^{-1}=\mathcal{O}_{\tilde{D}}
\]
Hence $\pi^{\ast}L \cong \mathcal{O}_{\tilde{D}}$. Hence $L \otimes \pi_{\ast}\mathcal{O}_{\tilde{D}} \cong \pi_{\ast}\mathcal{O}_{\tilde{D}}$ and therefore $L^{\otimes 2} \otimes \wedge^2 \pi_{\ast}\mathcal{O}_{\tilde{D}} \cong \wedge^2\pi_{\ast}\mathcal{O}_{\tilde{D}}$. Since $\wedge^2 \pi_{\ast}\mathcal{O}_{\tilde{D}}$ is invertible, this implies that $L^{\otimes 2} \cong \mathcal{O}_D$, as claimed.

It remains to show the second part of the theorem. By Theorem~\ref{artin}, there is an \'etale cover $\{f_i \colon U_i \rightarrow X \}$ containing $D$ such that $U_i=U_{i,1}\cup U_{i,2}$ is simple normal crossing, with two smooth irreducible components. Consider the pullback diagram
\[
\xymatrix{
\tilde{U}_i \ar[d]_{\tilde{f}_i}\ar[r]^{\pi_i} & U_i\ar[d]^{f_i} \\
\tilde{X} \ar[r]^{\pi} & X
}
\]
where $\pi \colon \tilde{X} \rightarrow X$ is the normalization of $X$. Then $\tilde{U}_i$ is the normalization of $U_i$ and therefore $\tilde{U}_i = U_{i,1} \coprod U_{i,2}$. Moreover, $\tilde{f}_i \colon \tilde{U}_i \rightarrow \tilde{X}$ is an \'etale cover of $X$ containing $\tilde{D}$. Let $D_i=f_i^{-1}(D)$ and $\tilde{D}_i=\pi_i^{-1}(D_i)=\tilde{f}_i^{-1}(\tilde{D})$. Then since $D_i$ is the singular locus of $U_i$, it follows that $D_i= U_{i,1}\cap U_{i,2}$. Hence $\tilde{D}_i = D_i \coprod D_i$. Moreover, the involution $\varepsilon$ of $\tilde{D}$ over $D$ lifts to an involution $\varepsilon_i$ of $\tilde{D}_i$ over $D_i$, interchanging the two irreducible components. Now by flat base change it follows that
\begin{equation}\label{1.2}
f_i^{\ast}\bigwedge^2 \pi_{\ast} \mathcal{O}_{\tilde{D}}(\tilde{D})=\bigwedge^2(\pi_i)_{\ast} \mathcal{O}_{\tilde{D}_i}(\tilde{D}_i)
\end{equation}
Now $(\pi_i)_{\ast} \mathcal{O}_{\tilde{D}_i}(\tilde{D}_i)=\mathcal{N}_{D_i/U_{i,1}} \oplus \mathcal{N}_{D_i/U_{i,2}}$ and hence 
\[
\bigwedge^2 (\pi_i)_{\ast} \mathcal{O}_{\tilde{D}_i}(\tilde{D}_i)=\mathcal{N}_{D_i/U_{i,1}} \otimes \mathcal{N}_{D_i/U_{i,2}}
\]
Then 
\[
(\pi_i)^{\ast} \bigwedge^2 (\pi_i)_{\ast} \mathcal{O}_{\tilde{D}_i}(\tilde{D}_i)=(\pi_i)^{\ast} (\mathcal{N}_{D_i/U_{i,1}} \otimes \mathcal{N}_{D_i/U_{i,2}})=\mathcal{O}_{\tilde{D}_i}(\tilde{D}_i) \otimes 
\varepsilon_i^{\ast} \mathcal{O}_{\tilde{D}_i}(\tilde{D}_i)
\]
and from~(\ref{1.2}) it follows that
\[
\tilde{f}_i^{\ast}\left(\pi^{\ast} \bigwedge^2 \pi_{\ast}\mathcal{O}_{\tilde{D}}(\tilde{D})\right) = \mathcal{O}_{\tilde{D}_i}(\tilde{D}_i) \otimes 
\varepsilon_i^{\ast} \mathcal{O}_{\tilde{D}_i}(\tilde{D}_i)=\tilde{f}_i^{\ast} 
(\mathcal{O}_{\tilde{D}}(\tilde{D}) \otimes \varepsilon^{\ast} \mathcal{O}_{\tilde{D}}(\tilde{D}))
\]
Now by glueing the above sheaves in the \'etale topology we get that
\[
\pi^{\ast} \bigwedge^2 \pi_{\ast}\mathcal{O}_{\tilde{D}}(\tilde{D})=\mathcal{O}_{\tilde{D}}(\tilde{D}) \otimes \varepsilon^{\ast} \mathcal{O}_{\tilde{D}}(\tilde{D})
\]
The above formula together with part 1. gives the claimed formula.
\end{proof}

\begin{corollary}\label{dp-snc}
Let $X=\cup_{i=1}^N X_i $ be simple normal crossing with only double point singularities. Then 
\[
T^1(X)=\mathcal{N}_{D_1/X_1}\otimes \cdots \otimes \mathcal{N}_{D_N/X_N}
\]
where $D_i = X_i \cap D$.
\end{corollary}

\begin{proof}

This follows immediately from Theorem~\ref{1.1} and the observation that since $X$ is simple normal crossing, $\pi \colon \tilde{D} \rightarrow D$ is the trivial \'etale cover and since $\pi^{\ast}(\omega_X \otimes \omega_D)=\mathcal{O}_{\tilde{D}}$, then $\omega_X \otimes \omega_D=\mathcal{O}_{D}$. 

\end{proof}

\begin{remark}
The formula in the above Corollary is a generalization of the corresponding well known formula when $X=X_1\cup X_2$, and $X$ is a divisor into a smooth variety $Y$~\cite{Fr83}.
\end{remark}

\section{The general case.}
Let $X$ be a scheme with normal crossing singularities. One major difference between the case when $X$ has singularities of multiplicity at most 2 and the general case, is that unlike the double point case when the singular locus $D$ of $X$ is smooth, in general $D$ may not even be Cohen-Macauley. For this reason it is preferable to work with a smooth model of $D$ instead of $D$ itself. In this section we will give a formula for $T^1(X)$ in a suitable log resolution of the pair $(X,D)$, in the case when $\dim X \leq 3$. The two dimensional case is a special case of a much more general result in~\cite{Tzi09}.

\begin{definition}
Let $X$ be a scheme with normal crossing singularities. We denote by $X^{max}$ the set of points having maximal multiplicity and by $X^{\geq s}$ the set of points of multiplicity at least $s$. A straightforward local calculation shows that $X^{max}$, $X^{\geq s}$ are closed subschemes of $X$ and that $X^{max}$ is smooth.
\end{definition}

The next theorem gives a formula for $T^1(X)$ in a suitable log resolution of $(X,D)$.

\begin{theorem}\label{2}
Let $X$ be a quasi-projective scheme with normal crossing singularities and let $D$ be its singular locus. Construct inductively the sequence of morphisms
\[
X^{\prime}=X_k \stackrel{f_k}{\la} X_{k-1} \stackrel{f_{k-1}}{\la} \cdots \stackrel{f_2}{\la} X_1 \stackrel{f_1}{\la} X_0=X
\]
such that $X_{i+1} $ is the blow up of $X_i$ along $X_i^{max}$. Let $f=f_1 \circ \cdots \circ f_k$. Let $D^{\prime} \subset X^{\prime}$ be the divisorial part of $f^{-1}(D)$ that dominates $D$ and $E_s$ the reduced $f$-exceptional divisor that dominates $X^{\geq s}$, $s\geq 3$. Then $X^{\prime}$ and $D^{\prime}$ are smooth and
\[
f^{\ast}T^1(X)=\mathcal{O}_{D^{\prime}}(D^{\prime}) \otimes \varepsilon^{\ast} \mathcal{O}_{D^{\prime}}(D^{\prime}) \otimes (\otimes_{s \geq 3} \mathcal{O}_{X^{\prime}}(sE_s))
\]
where $\varepsilon$ is the unique nontrivial involution of $D^{\prime}$ over $\bar{D}$, where $\bar{D}$ is the normalization of $D$. 
\end{theorem}

\begin{proof}
We proceed in two steps.

\textbf{Step 1.} First we succesively blow up the locus of points of multiplicity $\geq 3$ in order to reduce the calculation to that of a normal crossing scheme with only double points. Let $m$ be the maximal multiplicity of the singularities of $X$. If $m=2$, then go to step 2. Suppose that $m\geq 3$. Then locally at a point of maximal multiplicity, $X$ is given by $x_1\cdots x_m=0 \subset \mathbb{C}^{n+1}$, and $X^{max}$ by $x_1=x_2=\cdots =x_m=0$. Let $f_1 \colon X_1 \la X$ be the blow up of $X$ along $X^{max}$, $B_1$ the $f_1$-exceptional divisor and $D_1=(f_1)_{\ast}^{-1}D$. A straightforward local calculation shows that $X_1$ has normal crossing singularities of maximal multiplicity $m_1=m-1$, $B_1$ is not contained in the singular locus of $X_1$ and that $X_1^{max} \not \subset B_1$. Repeating the previous process of blowing up the locus of highest multiplicity, we get a sequence of maps
\begin{equation}\label{th2.1}
X_{m-2}\stackrel{f_{m-2}}{\la} \cdots \stackrel{f_{2}}{\la} X_1 \stackrel{f_{1}}{\la} X
\end{equation}
where $X_{m-2}$ has only normal crossing points of multiplicity 2, and hence its singular locus $D_{m-2}$ is smooth. By a slight abuse of notation, denote by $B_i$ the birational transform of the $f_i$ exceptional divisor in $X_{m-2}$ and let $g = f_{m-2}\circ \cdots \circ f_1$. Then,
\[
g^{\ast}T^1(X)=T^1(X_{m-2}) \otimes L
\]
where $L$ is a line bundle on $D_{m-2}$ of the form $\otimes_i \mathcal{O}_{X_{m-2}}(k_iB_i) \otimes \mathcal{O}_{D_{m-1}}$. We proceed to find the exact values of the $k_i$. This is local around the singularities of $X$. So we may assume that $X$ has an embeding as a divisor in a smooth $n+1$-dimensional variety $Y$, where $n=\dim X$. Then there is a sequence of birational maps
\begin{equation}\label{th2.2}
Y_{m-2}\stackrel{f_{m-2}}{\la} \cdots \stackrel{f_{2}}{\la} Y_1 \stackrel{f_{1}}{\la} Y
\end{equation}
where $Y_i$ is obtained from $Y_{i-1}$ by blowing up the locus of highest multiplicity of $X_{i-1} \subset Y_{i-1}$. Let $F_i$ be the $f_i$-exceptional divisor. Then $X_i=(f_{i})^{-1}_{\ast}X_{i-1}$ and $B_i=F_i \cdot X_i$ and,
\begin{equation}\label{eq2.1}
f_1^{\ast} \mathcal{O}_Y(X)=\mathcal{O}_{Y_1}(X_1) \otimes \mathcal{O}_{Y_1}(mB_1)
\end{equation}
Moreover, by Proposition~\ref{0}, $T^1(X)=\mathcal{O}_Y(X)\otimes \mathcal{O}_D$ and therefore~(\ref{eq2.1}) gives that,
\[
f_1^{\ast}T^1(X)=T^1(X_1)\otimes \mathcal{O}_{Y_1}(mB_1)
\]
Continuing in a similar fashion we get that,
\begin{equation}\label{eq2.2}
g^{\ast}T^1(X)=T^1(X_{m-2})\otimes (\otimes_{i=1}^{m-2} \mathcal{O}_{X^{\prime}}((m-i+1)B_i))
\end{equation}

\textbf{Step 2.} Let $f_{m-1} \colon X_{m-1} \la X_{m-2}$ be the normalization of $X_{m-2}$, which is smooth since $X_{m-2}$ has only double point singularities. Let $D_{m-2}$ its singular locus and $D_{m-1}=f_{m-1}^{-1}(D_{m-2})$. Then by Theorem~\ref{1},
\begin{equation}\label{eq2.3}
f_{m-1}^{\ast}T^1(X_{m-2})=\mathcal{O}_{D_{m-1}}(D_{m-1}) \otimes \varepsilon^{\ast} \mathcal{O}_{D_{m-1}}(D_{m-1})
\end{equation}
where $\varepsilon$ is the unique involution of $D_{m-1}$ over $D_{m-2}$ interchanging the fibers of $f_{m-1} \colon D_{m-1} \rightarrow D_{m-2}$. Let $f=f_{m-1}\circ g$ and $B^{\prime}_i=f_{m-1}^{-1}B_i$. Then from (\ref{eq2.2}),~(\ref{eq2.3}) it follows that
\begin{equation}\label{eq2.4}
f^{\ast}T^1(X)=    \mathcal{O}_{D_{m-1}}(D_{m-1}) \otimes \varepsilon^{\ast} \mathcal{O}_{D_{m-1}}(D_{m-1}) \otimes (\otimes_{i=1}^{m-2} \mathcal{O}_{X^{\prime}}((m-i+1)B^{\prime}_i))
\end{equation}
Note that by construction, $B^{\prime}_i$ dominates the locus of points of multiplicity $\geq m-i+1$. Now setting $X^{\prime}=X_{m-1}$, $E_{m-i+1}=B_i $,~(\ref{eq2.4}) takes the form stated in the theorem.
\end{proof}

\begin{remark}
One may try to get a formula for $T^1(X)$ in the normalization $\pi \colon \tilde{X} \la X$. However, $\tilde{D}=\pi^{-1}(D)$ is singular and it is preferable to work with smooth varieties. The pair $(D^{\prime},X^{\prime})$ is a log resolution for $(D,X)$ that is obtained in a natural way by repeatedly blowing up the centers of maximal multiplicity. The disadvantage of this approach is that $(D^{\prime},X^{\prime})$ is not characterized by any numerical property that would make it unique, as for example the minimal log resolution in the case of surfaces. However, in dimension at most 3 we can get a more natural description by using the minimal model program.
\end{remark}

\begin{theorem}\label{3}
Let $X$ be a quasi-projective scheme with normal crossing singularities with $\dim X \leq 3$ and let $D$ be its singular locus. Let $\phi \colon \tilde{X} \la X$ be a birational morphism such that
\begin{enumerate}
\item $(\tilde{X},1/2\tilde{D})$, $\tilde{D}$ are terminal;
\item $K_{\tilde{X}}+1/2\tilde{D}$ and $K_{\tilde{D}}$ are $\phi$-nef.
\end{enumerate}
where $\tilde{D} \subset \tilde{X}$ be the reduced divisorial part of $\phi^{-1}(D)$ that dominates $D$ (such spaces do exist). Then
\[
\phi^{\ast}T^1(X)=\mathcal{O}_{\tilde{D}}(\tilde{D}) \otimes \varepsilon^{\ast}\mathcal{O}_{\tilde{D}}(\tilde{D}) \otimes \mathcal{O}_{\tilde{X}}(3E)\otimes \mathcal{O}_{\tilde{D}}
\]
where $E\subset \tilde{X}$ is the reduced $\phi$-exceptional divisor that dominates the set of singular points of $X$ of multiplicity at least three and $\varepsilon$ the unique nontrivial involution of $\tilde{D}$ over the normalization $\bar{D}$ of $D$.
\end{theorem}
\begin{remark}
\begin{enumerate}
\item In the case of surfaces, $(\tilde{D},\tilde{X})$ is simply the minimal log resolution of $(D,X)$ and the theorem is a special case of~\cite{Tzi09}.
\item In the case of 3-folds,~(\ref{3}.1) implies that $\tilde{D}$ is smooth in $\tilde{X}$. However, the form written seems to be more natural to generalize in higher dimensions.
\item The proof of the theorem shows that if $\bar{D}$ is the normalization of $D$, then the induced map $\tilde{D} \la \bar{D}$ is \'etale.
\item The problems in order to obtain a similar result in all dimensions are the following. The existence of a pair $(\tilde{X},\tilde{D})$ with the properties stated in Theorem~\ref{3} is not a formal consequence of the minimal model program. One could start with a log resolution $(X^{\prime},D^{\prime})$ of $(X,D)$ and then try to run a MMP simultaneously for $(X^{\prime},1/2D^{\prime})$ and $D^{\prime}$, but this cannot be done in general. The construction of $(\tilde{X},\tilde{D})$ is explicit in the proof of Theorem~\ref{3} and in principle if one is careful enough it should be possible to generalize the argument in all dimensions.
\end{enumerate}
\end{remark}

\begin{proof}[Proof of Theorem~\ref{3}]
We only do the 3-fold case. The surface case is much simpler.

The proof consists of two steps. In the first one we will explicitly construct a pair $(\tilde{X},\tilde{D})$ with the properties of the statement and in the second part we will show that given any other pair $(X^{\prime},D^{\prime})$ having the same properties, $T^1(X)$ is given by the same formula.

\textbf{Step 1.} Since $\dim X=3$, $\mathrm{mult}_P(X) \leq 4$, for all $P\in X$. Now from the proof of Theorem~\ref{2}, there is a sequence of maps
\[
X^{\prime} \stackrel{g}{\la} X_2 \stackrel{f_2}{\la} X_1 \stackrel{f_1}{\la} X
\]
with the following properties. $X_1$ is the blow up of $X$ along the locus of points of multiplicity $4$. Then the singular points of $X_1$ have multiplicity at most $3$ and $X_2$ is the blow up of $X_1$ along the locus of points of multiplicity $3$. $X_2$ has only double points and $X^{\prime}$ is its normalization.

Let $E_1$ be the $f_1$-exceptional divisor and $D_1=(f_1)_{\ast}^{-1}D$. A straightforward local calculation shows that $D_1$ is the singular locus of $X_1$,  and over any singular point of $X$ with multiplicity $4$, $E_1 \cong (x_0x_1x_2x_3=0)\subset \mathbb{P}^3$. Moreover, $X_1^{max} \cap E_1 = \{P_1,P_2,P_3,P_4\}$, where $P_i$ are the vertices of the tetrahedron $x_0x_1x_2x_3=0$ in $\mathbb{P}^3$, and $E_1 \cap D_1 =\cup_{i,j} L_{i,j}$, where $L_{i,j}$ is the line connecting the verices $P_i$ and $P_j$.

Let $E_2$ be the $f_2$-exceptional divisor and $D_2=(f_2)_{\ast}^{-1}(D_1)$. Let $P\in X_1$ be a point such that $\mathrm{mult}_P(X_1)=3$. Then a straightforward local calculation shows that $X_2$ has only double points, $f_2^{-1}(P)\cong(x_0x_1x_2=0)\subset \mathbb{P}^2$, $D_2$ is the singular locus of $X_2$ and $D_2 \cap f_2^{-1}(P) =\{Q_1,Q_2,Q_3\}$, where $Q_i$ are the vertices of the triangle $x_0x_1x_2=0$ in $\mathbb{P}^3$.

Let $\tilde{E}_1=(f_2)_{\ast}^{-1}E_1$. Then $\tilde{E}_1=f_2^{\ast}E_1$. Moreover, $\tilde{E}_1$ is the blow up of $E_1$ along its vertices, and over any point of multiplicity $4$ of $X$, $\tilde{E}_1 \cap E_2 = \cup_{i=1}^4 f_2^{-1}(P_i)$. Moreover, from the previous discussion it follows that $f_2^{-1}(P_i)=F_{i,1}\cup F_{i,2}\cup F_{i,3}$, where $F_{i,j}$ are the edges of the triangle $(x_0x_1x_2=0) \subset \mathbb{P}^2$, and $\tilde{E}_1 \cap Z_2 = \cup_{i,j} \tilde{L}_{i,j}$, where $\tilde{L}_{i,j}$ are the birational transforms of the edges $L_{i,j}$ of $E_1$ in $E_2$.

The normalization $X^{\prime}$ of $X_2$ is the blow up of $X_2$ along $D_2$. Let $D^{\prime}=g^{-1}(D_2)$. Then $g \colon D^{\prime} \la D_2$ is \'etale of degree 2. Denote by $Z^{\prime}$ the birational transform of any divisor $Z \subset X_2$ in $X^{\prime}$. Then $E_2^{\prime}$, $E_1^{\prime}$ are the blow ups of $E_2$, $E_1$ along their edges and a straightforward local calculation shows that $E_2^{\prime}$, $E_1^{\prime}$ are smooth. Moreover, over any singular point $P\in X$ of multiplicity $4$, $E_2^{\prime}=\coprod_{i=1}^4 H_i $, where $H_i$ is isomorphic to the blow up of $\mathbb{P}^2$ along the vertices of the triangle $x_0x_1x_2=0$. In other words, $E^{\prime}_1$ is the log resolution of $E_1$ and its edges obtained by first blowing up the vertices and then the edges. Hence in each $H_i$ there is the following configuration of $f$-exceptional curves, where $f=f_1 \circ f_2 \circ g$.
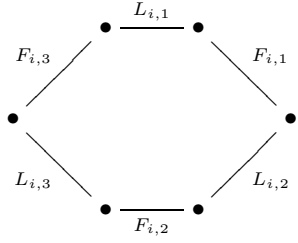
\[
\xymatrix{
              & \bullet   \ar@{-}[dl]_{F_{i,3}}\ar@{-}[r]^{L_{i,1}}                  &        \bullet \ar@{-}[dr]^{F_{i,1}}             &               \\
\bullet \ar@{-}[dr]_{L_{i,3}}       &                            &                             &      \bullet \ar@{-}[dl]^{L_{i,2}}            \\
              &     \bullet \ar@{-}[r]_{F_{i,2}}                &         \bullet              &        \\
}
\]
Moreover, we have the following intersection table.
\begin{gather}\label{intersection-table}
L_{i,j}^2=F_{i,j}^2=-1 \notag \\
\cup_{i,j}L_{i,j} = H_i \cdot D^{\prime}  \\
\cup_{i,j} F_{i,j} = H_i \cdot E_2^{\prime} \notag \\
L_{i,j} \cdot E^{\prime}_2=F_{i,j}\cdot D^{\prime} =2 \notag \\
L_{i,j}\cdot E_1^{\prime}=L_{i,j} \cdot (f_2\circ g)^{\ast} E_1 = (f_2\circ g)_{\ast}(L_{i,j}) \cdot E_1 =-1 \notag
\end{gather}

\begin{gather*}
F_{i,j} \cdot E_1^{\prime} =F_{i,j} \cdot (f_2\circ g)^{\ast} E_1 = (f_2\circ g)_{\ast}(F_{i,j}) \cdot E_1=0  \\
F_{i,j} \cdot E^{\prime}_2 =-1  \\
L_{i,j} \cdot D^{\prime} = (L_{i,j})^2_{E_1^{\prime}} =-1  \\
(L_{i,j})^2_{D^{\prime}}=L_{i,j} \cdot E_1^{\prime}=-1
\end{gather*}
where by $(L_{i,j})^2_{E_1^{\prime}}$, $(L_{i,j})^1_{D^{\prime}}$ we denote the self-intersection numbers of $L_{i,j}$ in $E_1^{\prime}$ and $Z^{\prime}$, respectively.

Next we use the Minimal Model Program to obtain a 3-fold with the properties stated in the theorem.

Standard adjunction formulas (which could be calculated using a local embedding $X \subset Y$ of $X$ in a smooth 4-fold $Y$ and following the argument of the proof of Theorem~\ref{1}), give that
\[
K_{X^{\prime}}+1/2D^{\prime} =f^{\ast}K_X-E_1^{\prime}-E_2^{\prime}-1/2D^{\prime}.
\]
and therefore
\begin{gather*}
(K_{X^{\prime}}+1/2D^{\prime})\cdot L_{i,j}=-1/2 \\
K_{X^{\prime}}\cdot L_{i,j} =0
\end{gather*}
Moreover, the exact sequence
\[
0 \la \mathcal{N}_{L_{i,j}/E_1^{\prime}} \la \mathcal{N}_{L_{i,j}/X^{\prime}} \la \mathcal{O}_{X^{\prime}}(E^{\prime}_1) \otimes \mathcal{O}_{L_{i,j}} \la 0
\]
give that
\[
\mathcal{N}_{L_{i,j}/X^{\prime}} = \mathcal{O}_{\mathbb{P}^1}(-1) \oplus \mathcal{O}_{\mathbb{P}^1}(-1)
\]
Hence there is a flopping contraction $h^{\prime} \colon X^{\prime} \la S$ contracting each $L_{i,j}$ to an ordinary 3-fold double point. Let $\psi \colon  X^{\prime} \dasharrow X^{\prime\prime} $ be its flop. This is a standard flop and its construction is described in the following diagram:
\begin{equation}\label{flop}
\xymatrix{
              &             W \ar[dl]_{\psi^{\prime}}\ar[dr]^{\psi^{\prime\prime}}               &              \\
X^{\prime} \ar[dr]_{h^{\prime}} \ar@{-->}[rr]^{\psi}         &                            &     X^{\prime\prime} \ar[dl]^{h^{\prime\prime}}\\
                                         &              S                 &      \\
}
\end{equation}
Here $W$ is the blow up of $X^{\prime}$ along $L_{i,j}$. Let $B$ be the $\psi^{\prime}$-exceptional divisor. Then over a neighborhood of any of the $L_{i,j}$, $B$ is a ruled surface over $L_{i,j}$ and in fact $B \cong \mathbb{P}^1 \times \mathbb{P}^1$. Then $\psi^{\prime\prime}$ is the contraction of the other ruling and hence $X^{\prime\prime}$ is also smooth. Let $L_{i,j}^{\prime\prime}$ be the $h^{\prime\prime}$-exceptional curves and $D^{\prime\prime}$ be the birational transform of $D^{\prime}$ in $X^{\prime\prime}$. Then from diagram~(\ref{flop}) and the explicit description of $X^{\prime}$, $E_1^{\prime}$ and $D^{\prime}$ it follows that $D^{\prime}_W=(\psi^{\prime})_{\ast}^{-1} D^{\prime} \cong D^{\prime}$ and therefore
\[
\phi \colon D^{\prime} \la D^{\prime\prime}
\]
is a morphism and is in fact, from~(\ref{intersection-table}), it is the contraction of the $(-1)$-curves of $D^{\prime}$. Then $L^{\prime\prime}_{i,j} \cdot D^{\prime\prime}=1$, $K_{X^{\prime\prime}}\cdot L^{\prime\prime}_{i,j}=0$ and therefore
\[
(K_{X^{\prime\prime}}+1/2 D^{\prime\prime})=1/2
\]
However, $K_{X^{\prime\prime}}+1/2 D^{\prime\prime}$ is still not nef. Let $F^{\prime\prime}_{i,j}$ be the birational transform of $F_{i,j}$ in $X^{\prime\prime}$. Then, since
\[
(\psi^{\prime\prime})^{\ast}D^{\prime\prime}=(\psi^{\prime})_{\ast}^{-1}D^{\prime}=(\psi^{\prime})^{\ast}D^{\prime}-B
\]
it follows from~(\ref{intersection-table}) that
\[
D^{\prime\prime}\cdot F^{\prime\prime}_{i,j} = D^{\prime}\cdot F_{i,j} - B \cdot (\psi^{\prime})_{\ast}^{-1} F_{i,j} =2-2=0
\]
Moreover, $K_{X^{\prime\prime}} \cdot F^{\prime\prime}_{i,j} =K_{X^{\prime}}\cdot F_{i,j} = -1$ and hence
\[
(K_{X^{\prime\prime}}+1/2 D^{\prime\prime}) \cdot F^{\prime\prime}_{i,j} = -1
\]
Let $H^{\prime\prime}_i=\phi_{\ast}^{-1} H_i$. Then from diagram~\ref{flop} it follows that $H_i^{\prime\prime} \cong \mathbb{P}^2$. Let us elaborate more on this. As mentioned earlier in the proof, if $P \in X$ is any point of multiplicity $4$ of $X$, then $f_1^{-1}(P)$ is isomorphic to the tetrahedron $x_0x_1x_2x_3=0$ in $\mathbb{P}^3$. Let $\Delta_i$ be its faces, $i=1,2,3,4$. Then $\Delta_i \cong \mathbb{P}^2$ and the edges $L_{i,j}$ of $f_1^{-1}(P)$ on $\Delta_i$ is the triangle $x_0x_1x_2=0$. Then $H_i$ is the blow up of $\Delta_i$ along the vertices of the triangle and $F_{i,j}$ are the exceptional curves. Now $H^{\prime\prime}_{i}$ is the contraction of the $L_{i,j}$. But this is exactly the construction of the standard quadratic transformation of $\mathbb{P}^2$. hence $H_i^{\prime\prime} \cong \mathbb{P}^2$ and
\[
\Delta_i=\mathbb{P}^2 \dasharrow \mathbb{P}^2\cong H_i^{\prime\prime}
\]
is the standard quadratic transformation of $\mathbb{P}^2$.

Hence $E_1^{\prime\prime}$ is a disjoint union of projective planes. Moreover,
\[
(\psi^{\prime\prime})^{\ast}E_1^{\prime\prime}=(\psi^{\prime})^{\ast}E_1^{\prime}-B
\]
and hence
\[
F_{i,j}^{\prime\prime} \cdot E_1^{\prime\prime}=F_{i,j}\cdot E_1^{\prime}-F_{i,j}\cdot B =0-2=-2
\]
Therefore, there is a birational morphism $\alpha \colon X^{\prime\prime} \la \tilde{X} $ over $X$, contracting every irreducible component of $E_1^{\prime\prime}$ to a cyclic quotient singularity of type $\frac{1}{2}(1,1,1)$. let $\tilde{D}=\alpha_{\ast}D^{\prime\prime}$. Then $\tilde{D} \cong D^{\prime\prime}$ and from the construction it follows that it is contained in the smooth part of $\tilde{X}$. Moreover, $K_{\tilde{X}}+1/2\tilde{D}$ and $K_{\tilde{D}}$ are both nef over $X$. Hence $(\tilde{X},\tilde{D})$ satisfies the numerical properties of the theorem. It also follows from the above construction that the natural induced map $\tilde{D} \la \overline{D}$, where $\overline{D}$ is the normalization of $D$, is \'etale of degree 2.

Next we show that $T^1(X)$ is given by the formula claimed in the statement. We do it by moving around the diagram
\begin{equation}\label{construction-diagram1}
\xymatrix{
              &             W \ar[dl]_{\psi^{\prime}}\ar[dr]^{\psi^{\prime\prime}}               &              \\
X^{\prime} \ar[d]_{f} \ar@{-->}[rr]^{\psi}         &                            &     X^{\prime\prime} \ar[d]^{\alpha}\\
                  X                       &                                                  &   \tilde{X} \ar[ll]^{\phi}    \\
}
\end{equation}
and the corresponding for $D$,
\begin{equation}\label{construction-diagram2}
\xymatrix{
D^{\prime} \ar[rr]^{\psi}\ar[dr]_{f} & & D^{\prime\prime}=\tilde{D} \ar[dl]^{\phi} \\
                                     & D &
}
\end{equation}
Here we should remark that from the construction of diagram~(\ref{construction-diagram1}) it follows that $(\psi^{\prime})_{\ast}^{-1}D^{\prime} \cong D^{\prime}$ and hence $\psi \colon D^{\prime} \la D^{\prime\prime}$ is just $\psi^{\prime\prime} \colon (\psi^{\prime})_{\ast}^{-1}D^{\prime} \la D^{\prime\prime}$.

From Theorem~\ref{2}, it follows that
\begin{equation}\label{eq-3.1}
f^{\ast}T^1(X)=\mathcal{O}_{D^{\prime}}(D^{\prime})\otimes \varepsilon^{\ast}_1   \mathcal{O}_{D^{\prime}}(D^{\prime}) \otimes \mathcal{O}_{X^{\prime}}(4E_1^{\prime}) \otimes \mathcal{O}_{X^{\prime}}(3E_2^{\prime})
\otimes \mathcal{O}_{D^{\prime}}
\end{equation}
where $\varepsilon_1$ is the unique involution of $D^{\prime}$ interchanging the fibers of $g$. Next we reduce the previous formula to $D^{\prime\prime}=\tilde{D}$. First we get a formula in $X^{\prime\prime}$. It is not hard to see that the involution $\varepsilon_1$ induces an involution $\varepsilon_2$ of $D^{\prime\prime}$ over $\bar{D}$ that fits in a commutative diagram
\[
\xymatrix{
D^{\prime} \ar[r]^{\psi}\ar[d]_{\varepsilon_1} & D^{\prime\prime}\ar[d]^{\varepsilon_2} \\
D^{\prime} \ar[r]^{\psi} & D^{\prime\prime}
}
\]

From diagram~(\ref{construction-diagram1}) and standard adjunctions we get that
\begin{gather*}
\psi^{\ast}\mathcal{O}_{D^{\prime\prime}}(D^{\prime\prime})=\mathcal{O}_{D^{\prime}}(D^{\prime})\otimes \mathcal{O}_W(-B) \otimes \mathcal{O}_{D^{\prime}}\\
\psi^{\ast} ( \mathcal{O}_{X^{\prime\prime}}(E_2^{\prime\prime})\otimes \mathcal{O}_{D^{\prime\prime}})=( \mathcal{O}_{X^{\prime}}(E_2^{\prime})\otimes \mathcal{O}_{D^{\prime}})\otimes \mathcal{O}_W(2B) \otimes \mathcal{O}_{D^{\prime}}\\
\psi^{\ast} ( \mathcal{O}_{X^{\prime\prime}}(E_1^{\prime\prime})\otimes \mathcal{O}_{D^{\prime\prime}})=( \mathcal{O}_{X^{\prime}}(E_1^{\prime})\otimes \mathcal{O}_{D^{\prime}})\otimes \mathcal{O}_W(-B) \otimes \mathcal{O}_{D^{\prime}}\\
\end{gather*}
Hence
\[
(\phi \circ \alpha)^{\ast}T^1(X)=\mathcal{O}_{D^{\prime\prime}}(D^{\prime\prime}) \otimes \varepsilon_2^{\ast} \mathcal{O}_{D^{\prime\prime}}(D^{\prime\prime}) \otimes \mathcal{O}_{X^{\prime\prime}}(4E_1^{\prime\prime}) \otimes \mathcal{O}_{X^{\prime\prime}}(3E_2^{\prime\prime})
\otimes \mathcal{O}_{D^{\prime}}
\]
and so the formula is unchanged in $X^{\prime\prime}$ (we did not expect a change since $X^{\prime}$ and $X^{\prime\prime}$ are isomorphic in codimension 1). Moreover, a carefull look at the construction of $X^{\prime\prime}$ reveals that $E_1^{\prime\prime} \cap D^{\prime\prime} = \emptyset$. Hence
\[
(\phi \circ \alpha)^{\ast}T^1(X)=\mathcal{O}_{D^{\prime\prime}}(D^{\prime\prime}) \otimes \varepsilon_2^{\ast} \mathcal{O}_{D^{\prime\prime}}(D^{\prime\prime}) \otimes \mathcal{O}_{X^{\prime\prime}}(3E_2^{\prime\prime})
\otimes \mathcal{O}_{D^{\prime}}
\]
and therefore
\[
\phi^{\ast}T^1(X)=\mathcal{O}_{\tilde{D}}(\tilde{D}) \otimes \varepsilon^{\ast} \mathcal{O}_{\tilde{D}}(\tilde{D}) \otimes  \mathcal{O}_{\tilde{X}}(3E)
\otimes \mathcal{O}_{\tilde{D}}
\]
where $E=\alpha_{\ast}E_2^{\prime\prime}$ and $\varepsilon=\varepsilon_2$ is the unique nontrivial involution of $\tilde{D}$ over $\bar{D}$, as claimed in the statement.

\textbf{Step 2.} Let $\psi \colon (\hat{X},\hat{D}) \la (X,D)$ be a morphism such that $(\hat{X},1/2\hat{D})$ is terminal and  $K_{\hat{X}}+1/2\hat{D}$, $K_{\hat{D}}$ are $\psi$-nef. We will show that $\psi^{\ast} T^1(X)$ is given by the formula stated in the theorem.

Since $(\hat{X},1/2\hat{D})$ is terminal and nef over $X$, it follows that there is a birational map
\[
g \colon \tilde{X} \dasharrow \hat{X}
\]
which is an isomorphism in codimension 1. Moreover, since $K_{\hat{D}}$ and $K_{\tilde{D}}$ are nef over $D$, $g$ induces an isomorphism between $\tilde{D}$ and $\hat{D}$. Since we already know that $\tilde{D} \la \bar{D}$ is \'etale, $\tilde{Z}$ does not contain any $\phi$-exceptional curves. Therefore $\hat{Z}$ does not contain any $\psi$-exceptional curves and hence there is a commutative diagram
\[
\xymatrix{
\mathrm{Ref}(\tilde{X}) \ar[r]^{g_{\ast}}\ar[d] & \mathrm{Ref}(\hat{X}) \ar[d] \\
\mathrm{Pic}(\tilde{D}) \ar[r]^{g_{\ast}} & \mathrm{Pic}(\hat{D})
}
\]
where $\mathrm{Ref}(\tilde{X})$, $\mathrm{Ref}(\hat{X})$ are the groups of rank 1 reflexive sheaves on $\tilde{X}$ and $\hat{X}$, respectively. Moreover, the horizontal maps are isomorphisms and the vertical the restrictions (the restriction maps are well defined because $\tilde{X}$, $\hat{X}$ are terminal and $\tilde{D}$, $\hat{D}$ smooth).

Since $g$ is an isomorphism in codimension 1, $\phi$ and $\psi$ have the same exceptional divisors. Hence if $\tilde{E}$ is $\phi$-exceptional dominating the locus of points of multiplicity at least 3, $\hat{E}=g_{\ast}^{-1} \tilde{E}$ is $\psi$-exceptional dominating the set of points of multiplicity at least 3 and $g_{\ast}\mathcal{O}_{\tilde{X}}(\tilde{E})=\mathcal{O}_{\hat{X}}(\hat{E})$. Moreover, $g_{\ast}\mathcal{O}_{\tilde{X}}(\tilde{D})=\mathcal{O}_{\hat{X}}(\hat{D})$. Hence
\begin{gather*}
\psi^{\ast}T^1(X)=g_{\ast}\phi^{\ast}T^1(X)=g_{\ast}(\mathcal{O}_{\tilde{D}}(\tilde{D})  \otimes   \varepsilon^{\ast}\mathcal{O}_{\tilde{D}}(\tilde{D})  \otimes \mathcal{O}_{\tilde{X}}(3\tilde{E})
\otimes \mathcal{O}_{\tilde{D}})=\\
\mathcal{O}_{\hat{D}}(\hat{D}) \otimes \varepsilon^{\ast}\mathcal{O}_{\hat{D}}(\hat{D}) \otimes \mathcal{O}_{\hat{X}}(3\hat{E})
\otimes \mathcal{O}_{\hat{D}}
\end{gather*}
as claimed.
\end{proof}


\begin{thebibliography}{Lichtenbaum}
\bibitem[Art69]{Art69} M. Artin, \textit{Algebraic approximation of structures over complete local rings}, Publ. Math. IHES 36, 1969, 23-58.
\bibitem[Fr83]{Fr83} R. Friedman, \textit{Global smoothings of varieties with normal crossings}, Ann. of Math. 118, 1983, 75-114.
\bibitem[Fu90]{Fu90} T. Fujita, \textit{On Del Pezzo fibrations over curves}, Osaka J. Math. 27, 1990, 229-245.
\bibitem[KKMS73]{KKMS73} G.Kempf, F. Knudsen, D. Mumford, B. Saint Donat, \textit{Toroidal Embeddings I}, Lect. Notes Math. Vol. 339, Springer-Verlag, Berlin Heidelberg, 1973.
\bibitem[KSB88]{KSB88} J. Koll{\'{a}}r, N. I. Shepherd-Barron, \textit{Threefolds and deformations of surface singularities},
Invent. Math. 91, 1988, 299-338.
\bibitem[Li-Sch67]{Li-Sch67} S. Lichtenbaum, M. Schlessiger, \textit{ The cotangent complex of a morphism}, Trans. Amer. Math. Soc. 128, 1967, 41-70.
\bibitem[LNS05]{LNS05} J. Lipman, J. Nayak, S. Sastry, \textit{ Pseudofunctorial behavior
of Cousin complexes on formal schemes}, Variance and duality for Cousin complexes on formal schemes. Contemp. Math., 375, Amer. Math. Soc., Providence, 2005, 3-133.
\bibitem[LRT07]{LRT07} L. A. Tarr\'io, A. J. L\'opez, M. P. Rodr\'iguez, \textit{Infinitesimal lifting and Jacobi criterion for smoothness on 
formal schemes}, Communications in Algebra 35, 2007, 1341-1367.
\bibitem[Mi80]{Mi80} J. Milne, \textit{\'Etale Cohomology}, Princeton University Press, 1980.
\bibitem[Re94]{Re94} M. Reid, \textit{Nonnormal del Pezzo surfaces}, Publ. Res. Inst. Math. Sci. 30, 1994, no. 5, 695--727. 
\bibitem[Ser06]{Ser06} E. Sernesi, \textit{Deformations of Algebraic Schemes}, Springer-Verlag Berlin Heidelberg, 2006.
\bibitem[Sch68]{Sch68} M. Schlessinger, \textit{Functors of Artin rings}, Trans. AMS 130, 1968, 208-222.
\bibitem[Tzi09]{Tzi09} N. Tziolas, \textit{$\mathbb{Q}$-Gorenstein smoothings of nonnormal surfaces}, Amer. J. Math.  131  (2009),  no. 1, 171-193.
\end{thebibliography}
\end{document}